\documentclass[11pt]{article}
\usepackage{epic,latexsym,amssymb}
\usepackage{color}
\usepackage{tikz}

\newtheorem{thm}{Theorem}
\newtheorem{lem}[thm]{Lemma}

\newtheorem{cor}{Corollary}
\newtheorem{ob}{Observation}
\newtheorem{prop}{Proposition}

\newcommand{\qed}{$\Box$}

\newcommand{\Eh}{{Edge-hitter }}
\newcommand{\w}{{\rm w}}

\newcommand{\smallqed}{{\tiny ($\Box$)}}
\newcommand{\TR}[1]{\mbox{$\tau(#1)$}}

\newcommand{\St}{{Staller }}

\newcommand{\dsum}[1]{\sum_{w\in N(#1)} d_H(w)}

\def \nW {W_{_H}}

\def \nB {B_{_H}}
\def \nGr {G_{_H}}
\def \nH {n_{_H}}
\def \mH {m_{_H}}
\def \nHp {n_{_{H'}}}
\def \mHp {m_{_{H'}}}

\newenvironment{unnumbered}[1]{\trivlist \item [\hskip \labelsep {\bf
#1}]\ignorespaces\it}{\endtrivlist}

\newcommand{\1}{ \vspace{0.1cm} }

\newcommand{\cH}{{\cal H}}

\newcommand{\dstart}{\tau_g}
\newcommand{\sstart}{\tau_g^\prime}

\begin{document}

\title{Bounds on the Game Transversal Number
in Hypergraphs}

\author{$^{1,2}$Csilla Bujt\'{a}s,
$^3$Michael A. Henning,
 and $^{1,2}$Zsolt Tuza
\\ \\
$^1$Department of Computer Science and Systems Technology \\
University of Pannonia \\
H-8200 Veszpr\'{e}m, Egyetem u.\ 10, Hungary\\
\small \tt Email: bujtas@dcs.uni-pannon.hu, \enskip
  tuza@dcs.uni-pannon.hu \\
\\
$^2$Alfr\'ed R\'enyi Institute of Mathematics \\
       Hungarian Academy of Sciences \\
H-1053 Budapest, Re\'altanoda u.\ 13-15, Hungary \\
\\
$^3$Department of Mathematics \\
University of Johannesburg \\
Auckland Park, 2006 South Africa\\
\small \tt Email: mahenning@uj.ac.za  \\
}

\date{}
\maketitle

\begin{abstract}
Let $H = (V,E)$ be a hypergraph with vertex set $V$ and edge set $E$
of order $\nH = |V|$ and size $\mH = |E|$.  A transversal in $H$ is
a subset of vertices in $H$ that has a nonempty intersection with
every edge of $H$. A vertex hits an edge if it belongs to that edge.
The transversal game played on $H$ involves of two players,
\emph{Edge-hitter} and \emph{Staller}, who take turns choosing a
vertex from $H$. Each vertex chosen must hit at least one edge not
hit by the vertices previously chosen. The game ends when the set of
vertices chosen becomes a transversal in $H$. Edge-hitter wishes to
minimize the number of vertices chosen in the game, while Staller wishes
to maximize it. The \emph{game transversal number}, $\tau_g(H)$, of
$H$ is the number of vertices chosen when Edge-hitter starts the
game and both players play optimally.
We compare the game transversal number of a hypergraph
 with its transversal number, and also
 present an important fact concerning the monotonicity of $\tau_g$,
 that we call the Transversal Continuation Principle.
%
It is known that if $H$ is a hypergraph with all edges of size at
least~$2$,  and $H$ is not a $4$-cycle, then $\tau_g(H) \le
\frac{4}{11}(\nH+\mH)$;  and if $H$ is a (loopless) graph, then
$\tau_g(H) \le \frac{1}{3}(\nH + \mH + 1)$. We prove that if $H$ is
a $3$-uniform hypergraph, then  $\tau_g(H) \le \frac{5}{16}(\nH +
\mH)$, and if $H$ is  $4$-uniform, then $\tau_g(H) \le
\frac{71}{252}(\nH + \mH)$.
\end{abstract}


\section{Introduction}

In this paper, we continue the study of the transversal game in
hypergraphs which was first  investigated  in~\cite{BuHeTu15+}.
  The results obtained there implied the proof of the
$\frac{3}{4}$-Game Total Domination Conjecture, which was
posted by Henning,  Klav\v{z}ar and Rall~\cite{hkr-2015+}, over the
class of graphs with minimum degree at least~$2$.

Hypergraphs are systems of sets which are conceived as natural
extensions of graphs.  A \emph{hypergraph} $H = (V(H),E(H))$ is a
finite set $V(H)$ of elements, called \emph{vertices}, together with
a finite multiset $E(H)$ of nonempty subsets of $V(H)$, called
\emph{hyperedges} or simply \emph{edges}. If the hypergraph $H$ is
clear from the context, we may write $V = V(H)$ and $E = E(H)$. We
shall use the notation $\nH =|V(H)|$ and $\mH=|E(H)|$, and sometimes
just $n$ and $m$ without subscript if the actual $H$ need not be
emphasized, to denote the {order} and the {size} of $H$,
respectively.  We say that two edges in $H$ \emph{overlap} if
they intersect in at least two vertices. A   hypergraph is
\emph{linear} if it has no overlapping edges.

A $k$-\emph{edge} in $H$ is an edge of cardinality~$k$.  The
hypergraph $H$ is said to be $k$-\emph{uniform} if every edge of $H$
is a $k$-edge. Every loopless graph is a $2$-uniform hypergraph.
Thus graphs are special hypergraphs. The \emph{degree} of a vertex
$v$ in $H$, denoted by $d_H(v)$,
 is the number of edges of $H$ which contain $v$.
 The maximum degree among the vertices of $H$ is denoted by $\Delta(H)$.

Two vertices $x$ and $y$ of $H$ are \emph{adjacent} if there is an
edge $e$  of $H$ such that $\{x,y\}\subseteq e$. The
\emph{neighborhood} of a vertex $v$ in $H$, denoted $N_H(v)$ or
simply $N(v)$ if $H$ is clear from the context, is the set of all
vertices different from $v$ that are adjacent to $v$. A vertex in
$N(v)$ is a \emph{neighbor} of $v$.

A subset $T$ of vertices in a hypergraph $H$ is a \emph{transversal}
(also called  \emph{hitting set} or \emph{vertex cover} or
\emph{blocking set} in many papers) if $T$ has a nonempty
intersection with every edge of $H$. A vertex \emph{hits} or
\emph{covers} an edge if it belongs to that edge.  The
\emph{transversal number} $\TR{H}$ of $H$ is the minimum size of a
transversal in $H$.  In hypergraph theory the concept of transversal
is  fundamental and well studied. The major monograph \cite{Berge89}
of hypergraph theory gives a detailed introduction to this topic. We
refer to~\cite{BuHeTu12,BuHeTuYe14,DoHe14,HeLo12,HeLo14,HeYe13,
HeYe13b,LoWa13}
 for recent results and further references.

\paragraph{The Game Transversal Number.} The transversal game
belongs to the   growing family of \emph{competitive optimization}
graph and hypergraph games. Competitive optimization variants of
coloring~\cite{Bod91,DZ99,Gar81,KK09,KT94,TZ15},
list-colouring~\cite{BST07,Sch09,Zhu09}, matching~\cite{CKOW13},
domination~\cite{BKR10,KWZ-2013}, total
domination~\cite{hkr-2015,hkr-2015+}, disjoint
domination~\cite{BuTu15}, Ramsey theory~\cite{BGKMSW11,GHK04,GKP08},
and more~\cite{BKP14} have been extensively investigated.

The transversal game played on a hypergraph $H$ involves two
players, \emph{Edge-hitter} and \emph{Staller}, who take turns
choosing a vertex from $H$. Each vertex chosen must hit at least one
edge not hit by the vertices previously chosen. We call such a
chosen vertex a \emph{legal move} in the transversal game. The game
ends when the set of vertices chosen becomes a transversal in $H$.
Edge-hitter wishes to end the game with the smallest possible
  number of vertices
chosen, and Staller wishes to end the game with as many vertices
chosen as possible. The \emph{game transversal number} (resp.\
\emph{Staller-start game transversal number}), $\tau_g(H)$ (resp.\
$\tau_g'(H)$), of $H$ is the number of vertices
chosen when Edge-hitter (resp.\ Staller) starts the game and both
players play  optimally.

A \emph{partially covered hypergraph} is a hypergraph  together with
a declaration that some edges are already covered; that is, they
need not be covered in the rest of the game.
 Once an edge has been covered, it plays no role
in the remainder of the game and can be deleted from the partially
covered hypergraph, as can all isolated vertices.
   Therefore, after those deletions we obtain a hypergraph
    being equivalent, from the transversal game viewpoint,
    to the partially covered hypergraph
    from which it has been derived;
     we call it a \emph{residual hypergraph}.
  We will also say that the original hypergraph $H$,
before any move has been made in the game, is a residual
  (and also partially covered)
 hypergraph.


Given a hypergraph $H$ and a subset $S$ of edges of $H$,  we denote
by $H|S$ the residual hypergraph\footnote{In the context of games
 we prefer to use the notation $H|S$, although its edge set coincides
  with that of the hypergraph denoted by $H-S$ in many
   hypergraph-theoretic papers.}
   in which the edges contained in $S$ do not appear anymore.
 We use $\dstart(H|S)$ (resp.\ $\sstart(H|S)$)
to denote the number of turns remaining in the transversal game on
$H|S$ under optimal play when \Eh (resp.\ Staller) has the next
turn.

 We will use the standard notation $[k] = \{1,\ldots,k\}$.

\section{Known Results}

Let $H_1$ be the hypergraph with vertex set $V(H_1) =
\{x_1,x_2,x_3,y_1,y_2,y_4\}$  and edge set $E(H_1) =
\{\{x_1,x_2,x_3\},  \{y_1,y_2,y_3\}, \{x_1,y_1\}, \{x_2,y_2\},
\{x_3,y_3\}\}$. For $k \ge 1$, let $H_k$ consist of $k$
vertex-disjoint copies of $H_1$, and let $\cH = \{H_k \colon \, k
\ge 1\}$. The hypergraph $H_3 \in \cH$ is illustrated in
Figure~\ref{f:H3}.

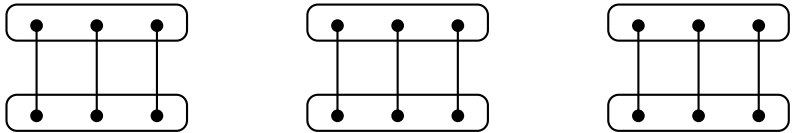
\begin{figure}[htb]
\begin{center}

\begin{tikzpicture}[scale=.8,style=thick,x=1cm,y=1cm]
\def\vr{2.5pt} 
\path (0,0) coordinate (x1);
\path (1,0) coordinate (x2);
\path (2,0) coordinate (x3);
\path (0,1.5) coordinate (y1);
\path (1,1.5) coordinate (y2);
\path (2,1.5) coordinate (y3);
%
\draw (x1) -- (y1);
\draw (x2) -- (y2);
\draw (x3) -- (y3);
\draw (x1) [fill=black] circle (\vr);
\draw (x2) [fill=black] circle (\vr);
\draw (x3) [fill=black] circle (\vr);
\draw (y1) [fill=black] circle (\vr);
\draw (y2) [fill=black] circle (\vr);
\draw (y3) [fill=black] circle (\vr);
\draw [rounded corners] (-0.5,-0.25)
rectangle (2.5,0.35) node [black,right] {}; \draw [rounded corners] (-0.5,1.25) rectangle (2.5,1.85) node [black,right] {}; %
\path (5,0) coordinate (x1);
\path (6,0) coordinate (x2);
\path (7,0) coordinate (x3);
\path (5,1.5) coordinate (y1);
\path (6,1.5) coordinate (y2);
\path (7,1.5) coordinate (y3);
%
\draw (x1) -- (y1);
\draw (x2) -- (y2);
\draw (x3) -- (y3);
\draw (x1) [fill=black] circle (\vr);
\draw (x2) [fill=black] circle (\vr);
\draw (x3) [fill=black] circle (\vr);
\draw (y1) [fill=black] circle (\vr);
\draw (y2) [fill=black] circle (\vr);
\draw (y3) [fill=black] circle (\vr);
\draw [rounded corners] (4.5,-0.25)
rectangle (7.5,0.35) node [black,right] {}; \draw [rounded corners] (4.5,1.25) rectangle (7.5,1.85) node [black,right] {}; %
%
\path (10,0) coordinate (x1);
\path (11,0) coordinate (x2);
\path (12,0) coordinate (x3);
\path (10,1.5) coordinate (y1);
\path (11,1.5) coordinate (y2);
\path (12,1.5) coordinate (y3);
%
\draw (x1) -- (y1);
\draw (x2) -- (y2);
\draw (x3) -- (y3);
\draw (x1) [fill=black] circle (\vr);
\draw (x2) [fill=black] circle (\vr);
\draw (x3) [fill=black] circle (\vr);
\draw (y1) [fill=black] circle (\vr);
\draw (y2) [fill=black] circle (\vr);
\draw (y3) [fill=black] circle (\vr);
\draw [rounded corners] (9.5,-0.25)
rectangle (12.5,0.35) node [black,right] {}; \draw [rounded corners] (9.5,1.25) rectangle (12.5,1.85) node [black,right] {}; %

\end{tikzpicture}
\end{center}
\vskip -0.25cm \caption{The hypergraph $H_3$ from the family~$\cH$.}
\label{f:H3}
\end{figure}

The following upper bound on the game transversal number of a hypergraph is established
  in~\cite{BuHeTu15+}.

\begin{thm}{\rm (\cite{BuHeTu15+})}
If $H$ is a hypergraph with all edges of size at least~$2$, and $H \ncong C_4$, then $\tau_g(H) \le \frac{4}{11}(\nH+\mH)$, with equality if and only if $H\in \cH$.
\label{thm1}
\end{thm}

As a special case of more general results due to Tuza~\cite{Tu90}
and Chv\'{a}tal and McDiarmid~\cite{ChMc92},  if $H$ is a simple
graph, then $\tau(H) \le \frac{1}{3}(\nH + \mH)$. This bound is
almost true for the game transversal number, as proved
 in~\cite{BuHeTu15+}.

\begin{thm}{\rm (\cite{BuHeTu15+})}
If $H$ is a $2$-uniform hypergraph, then $\tau_g(H) \le \frac{1}{3}(\nH+\mH+1)$.
\label{2uniform}
\end{thm}

\section{Main Results}

Since the transversal game played in a hypergraph $H$ ends when the
set  of vertices chosen becomes a transversal in $H$,
 it is clear  that
$\tau(H) \le \dstart(H)$ and $\tau(H) \le \sstart(H)$. If \Eh fixes
a minimum transversal set, $T$, in $H$ and adopts the strategy in
each of his turns to play a vertex from $T$ if possible, then he
guarantees that the game ends in no more than $2\tau(H)-1$ moves in
the Edge-hitter-start transversal game and in no more than
$2\tau(H)$ moves in the Staller-start transversal game. We state
this fact formally as follows.

\begin{ob}
For every hypergraph $H$, the following holds. \\
\indent {\rm (a)} $\tau(H) \le \dstart(H) \le 2\tau(H)-1 $. \\
\indent {\rm (b)} $\tau(H) \le \sstart(H) \le 2\tau(H)$.
\label{obser1}
\end{ob}

 It is easy to see that the equalities
$\tau(H) = \dstart(H) = \sstart(H)$ hold if $H$  is the disjoint
union of complete $k$-uniform hypergraphs. Further,
$\tau_g(H)=2\tau(H)-1$ and $\sstart(H)=2\tau(H)$ are valid if
$\tau(H) = 1$ and $H$ contains at least two different edges. In
Section~\ref{S:family} we present an infinite family of hypergraphs
$H$ with $\tau_g(H) = 2\tau(H)-1$ and $\tau_g'(H)= 2\tau(H)$. These
show that the lower and upper bounds given in Observation
\ref{obser1} cannot be improved even when $\tau(H)$ is large.

We next present a  simple but fundamental and widely applicable
 lemma, named the \emph{Transversal Continuation
Principle}, that expresses the monotonicity of $\dstart$ and
 $\sstart$ with respect to subhypergraphs.
 Its proof is given in Section~\ref{S:TCP}.

\begin{lem}
\label{lem:continuation}
{\rm (Transversal Continuation Principle)}
Let $H$ be a hypergraph and let $A,B \subseteq E(H)$.  If $B \subseteq A$, then $\dstart(H|A) \le \dstart(H|B)$ and $\sstart(H|A) \le \sstart(H|B)$.
\end{lem}

Let us mention without quoting the formal definitions that in any graph $G$ the
 dominating sets are in one-to-one correspondence with the transversals of the
 hypergraph whose edges are the closed neighborhoods of the vertices in $G$;
 and similarly, the total dominating sets in a graph without isolated vertices
  are in one-to-one correspondence with the transversals of the
   open neighborhoods of the vertices.
 (These facts are immediate by definition.)
In this way our Transversal Continuation Principle includes, as particular
 cases, the assertions called `Continuation Principle' for the domination
 game in \cite{KWZ-2013} and for the total domination game in \cite{hkr-2015},
 hence putting them on a higher level of generality.

As another consequence of the Transversal Continuation Principle,  the
number of moves in the Edge-hitter-start transversal game and the
Staller-start transversal game when played optimally can differ by
at most one. We state this formally as follows.

\begin{thm}
\label{thm:difference1}
For every hypergraph $H$, we have $|\dstart(H) - \sstart(H)| \le 1$.
\end{thm}

We remark that the hypergraphs that achieve equality in the bound of
Theorem~\ref{thm1},  namely the hypergraphs that belong to the
family~$\cH$, contain both $2$-edges and $3$-edges. Our  two main
results in this paper show that the upper bound of
Theorem~\ref{thm1} can be improved for $3$-uniform and $4$-uniform
hypergraphs as follows.

\begin{thm}
If $H$ is a $3$-uniform hypergraph, then  $\tau_g(H) \le \frac{5}{16}(\nH + \mH)$.
\label{3uniform}
\end{thm}

\begin{thm}
If $H$ is a $4$-uniform hypergraph, then  $\tau_g(H) \le
\frac{71}{252}(\nH + \mH)$. \label{4uniform}
\end{thm}

Proofs of Theorem~\ref{thm:difference1}, Theorem~\ref{3uniform} and
Theorem~\ref{4uniform} are given in Section~\ref{S:TCP},
Section~\ref{S:3uniform} and Section~\ref{S:4uniform}, respectively.

\section{Family of Hypergraphs}
\label{S:family}

By Observation~\ref{obser1}, every hypergraph $H$ satisfies
$\dstart(H) \le 2\tau(H)-1$ and $\sstart(H) \le 2\tau(H)$. In this
section, we present an infinite family of hypergraphs $H$ with
$\dstart(H) = 2\tau(H)-1$ and  $\sstart(H)= 2\tau(H)$. For this
purpose, we define a \emph{$k$-corona} of a hypergraph $H$  to be a
hypergraph obtained by attaching $k$ hyperedges (each of size at
least~$2$) to each vertex of $H$, where the hyperedges attached to a
vertex $v \in V(H)$ contain only degree-$1$ vertices apart from~$v$.

\begin{prop}
For every positive integer $k$ and for every hypergraph $H$ of order
at most $2^{k-1}-1$, every $k$-corona $H^k$ of $H$ satisfies
$\tau_g(H^k) = 2\tau(H^k)-1$ and $\sstart(H^k) = 2\tau(H^k)$.
\end{prop}
\begin{proof}
Let us denote the vertices of $H$ by $v_1, \dots, v_n$ and the
hyperedges attached to $v_i$ by $e(1,i), \dots, e(k,i)$. By our
assumption, $n <2^{k-1}$. Since $\tau(H^k)=n$, the inequalities
$\tau_g(H^k)\le 2n-1$ and $ \sstart(H^k) \le 2n$ are valid by
Observation~\ref{obser1}.
  Therefore, it suffices to prove that
Staller has a strategy to achieve at least $2n-1$ turns if
Edge-hitter starts the game, and at least $2n$ turns if Staller
starts.

First, we associate a weight $\w(e)$ with each hyperedge $e$ of
$H^k$ as follows. If $e$ is a hyperedge of $H$, then we let $\w(e) =
0$. If $e = e(j,i)$ is a hyperedge attached to $H$ for some $j \in
[k]$ and $i \in [n]$, then we let $\w(e) = 2^{j-1}$. As the game is
played, when a hyperedge is hit by a played vertex, the weight of
such a hyperedge becomes zero. Hence, if $\w(H^k)$ denotes the sum
of the weights of the edges in the residual hypergraph $H^k$, then
the game starts with
\[
\w(H^k) = \sum_{i=1}^n \sum_{j=1}^k 2^{j-1} = n(2^k-1),
\]
and is completed when $\w(H^k) = 0$; that is, the game is completed when the sum of the weights of the edges in the residual hypergraph $H^k$ equals zero. We consider the following strategy of Staller.

\noindent \textbf{Staller's Rule:} She always plays a vertex of degree~$1$ such that the incident hyperedge has the smallest positive weight in the residual hypergraph.

We show that if \St applies this rule, each of her moves together
with the next move of \Eh decreases the weight by at most~$2^{k}$.
If Staller plays a vertex incident to an attached hyperedge of
weight $2^s$ for some $s \in [k-1]\cup \{0\}$, then in the next turn
\Eh cannot choose a vertex which is incident to a hyperedge of
smaller positive weight. Moreover, no single vertex of $H^k$ is
incident with two hyperedges of the same positive weight. Hence,
Edge-hitter's move decreases the weight of the residual hypergraph
by at most
\[
\sum_{i=s}^{k-1} 2^{i}=2^{k}-2^s\,,
\]
and, together with Staller's previous move which decreases the weight by~$2^s$, their two moves combined decrease the weight by at most~$2^k$.

If \Eh begins the game, his first move decreases the weight of the residual hypergraph by at most
\[
\sum_{i=1}^{k} 2^{i-1}=2^{k}-1,
\]

\noindent while, if the weight of the residual hypergraph is not
zero after Edge-hitter plays his $(n-1)$st move (that is, after the
$(2n-3)$rd turn), then Staller's $(n-1)$st move in the $(2n-2)$nd
turn decreases the weight by at most $2^{k-1}$.  Therefore, the
weight of the residual hypergraph after the $(2n-2)$nd  turn is at
least
\[
n(2^k-1)-(2^{k}-1)-(n-2)2^k - 2^{k-1}= 2^{k-1}-n+1\,,
\]

\noindent which is at least 2, as we supposed $n \le 2^{k-1}-1$.
Since the obtained  hypergraph has still positive weight, there
exist some uncovered edges. Thus, \St has a strategy which makes
sure that the game is not complete after the $(2n-2)$nd turn,
implying that $\dstart(H^k) \ge 2n-1$.  Consequently,  $\dstart(H^k)
= 2n-1 = 2\tau(H^k)-1$.

Similarly, if Staller begins the game, then after her $n$th move played in the $(2n-1)$st turn, the weight of the residual hypergraph is at least
\[
n(2^k-1) - (n-1)2^k - 2^{k-1}= 2^{k-1}-n \ge 1\,.
\]

\noindent
 Thus, \St has a strategy which guarantees that
 the length of the game is at least $2n$,  implying that $\sstart(H^k) \ge 2n$.
 Therefore,  $\sstart(H^k) = 2n = 2\tau(H^k)$.~\qed
\end{proof}

\section{The Transversal Continuation Principle}
\label{S:TCP}

In this section, we present a proof of the Transversal Continuation Principle. Recall its statement.

\noindent \textbf{Lemma~\ref{lem:continuation}} {\rm (Transversal Continuation Principle)}. \emph{Let $H$ be a hypergraph and let $A,B \subseteq E(H)$.  If $B \subseteq A$, then $\dstart(H|A) \le \dstart(H|B)$ and $\sstart(H|A) \le \sstart(H|B)$.}

\medskip
\begin{proof} Two games will be played, Game A on the hypergraph $H|A$ and Game B on the hypergraph $H|B$. The first of these will be the real game, while Game B
will only be imagined by Edge-hitter. In Game A, \St will play optimally while in Game B, \Eh will play optimally.

We claim, by induction on the number of moves played,  that in each stage of the games, the set of edges that are covered in Game~B is a subset of the edges that are covered in Game~A. Since $B \subseteq A$, this is true at the start of the games. Suppose now that \St has (optimally) selected vertex $u$ in Game~A. This move of \St hits at least one new edge, say $e_u$, in Game~A. By the induction assumption, the edge $e_u$ is not yet hit in Game~B, and so the vertex $u$ is a legal move in Game~B. \Eh now copies the move of \St and plays vertex $u$ in Game~B, and then replies with an optimal move in Game~B. If this move is legal in Game A, \Eh plays it in Game~ A as well. Otherwise, if the game is not yet over, \Eh plays any other legal move in Game A. In both cases the claim assumption is preserved, which by induction also proves the claim.

We have thus proved that Game~A finishes no later than Game~B. Suppose thus that $k$ moves are played in Game~B. Since \Eh was playing optimally in Game~B, $k \le \dstart(H|B)$. On the other hand, because \St was playing optimally in Game~A and \Eh has a strategy to finish the game in $k$ moves, $\dstart(H|A) \le k$. Therefore, $\dstart(H|A) \le k \le \dstart(H|B)$. Thus, if \Eh is the first to play, the desired bound holds. In the above arguments we did not assume who starts first, hence in both cases Game~A will finish no later than Game~B, implying that $\sstart(H|A) \le \sstart(H|B)$.~\qed
\end{proof}

\medskip

  If two vertices are incident with precisely the same edges, then
   at most one of them can be played during the game.
 Now, assume that in the  residual
hypergraph $H$, vertex $v$ hits all the edges that $u$ hits, but
$d_H(v) > d_H(u)$. As a consequence of the Transversal Continuation
Principle, we may suppose that \Eh never plays  $u$ and \St never
plays  $v$.
%

Theorem~\ref{thm:difference1} follows from the Transversal Continuation Principle. Recall its statement.

\noindent \textbf{Theorem~\ref{thm:difference1}}. \emph{For every hypergraph $H$, we have $|\dstart(H) - \sstart(H)| \le 1$.}

\medskip
\begin{proof} Consider the Edge-hitter-start transversal game and let $v$ be the first move of Edge-hitter. Let $A$ be the set of edges hit by $v$ and let $B=\emptyset$, and consider the partially covered hypergraphs $H|A$ and $H|B$. We note that $H|B = H$ and $\dstart(H) = 1 + \sstart(H|A)$. By the Transversal Continuation Principle, $\sstart(H|A) \le \sstart(H|B) = \sstart(H)$. Therefore, $\dstart(H) =  \sstart(H|A) + 1 \le \sstart(H) + 1$. Analogously, $\sstart(H) \le \dstart(H) + 1$.~\qed
\end{proof}

\section{Proof of Theorem~\ref{3uniform} and Theorem~\ref{4uniform}}

We remark that if $H$ is a hypergraph, and $H'$ is obtained from $H$
by deleting all multiple edges in $H$  (in the sense that if $H$ has
$\ell$ distinct edges $e_1, e_2, \ldots, e_\ell$ that are multiple
edges, and so $e_1 = e_2 = \cdots = e_\ell$, then we delete $\ell -
1$ of these multiple edges), then $\tau_g(H') = \tau_g(H)$. Hence,
it suffices to prove Theorem~\ref{3uniform} and
Theorem~\ref{4uniform} in the case of hypergraphs with no multiple
edges.

\subsection{Proof of Theorem~\ref{3uniform}}
\label{S:3uniform}

 In this section, we prove Theorem~\ref{3uniform}. For
this purpose, we define a \emph{colored hypergraph} with
respect to the played vertices in the set $D$ as a hypergraph in
which every vertex is colored with one of four colors, namely white,
green, blue, or red, according to the following rules.
\begin{itemize}
\item A vertex is colored \emph{white} if it is incident with  at least~$3$  edges uncovered by $D$.
\item A vertex is colored \emph{green} if it is incident with  exactly~$2$   edges uncovered by $D$..
\item A vertex is colored \emph{blue} if it is incident with exactly~$1$ edge uncovered by $D$.
\item A vertex is colored \emph{red} if it is not incident with any
edges uncovered by $D$.
\end{itemize}

Further, an edge is colored \emph{white} if it is not covered by a
vertex of $D$, and is colored \emph{red} otherwise. Thus, an edge is
colored red if it contains a red vertex.

 By our definition given in the Introduction,  the residual hypergraph does not contain  red vertices and  red edges.
  That is, the vertices  of the residual hypergraph are colored with white, green and blue as defined
 above. Note that every edge of the residual hypergraph is white.

 In a colored hypergraph, and also in a colored residual hypergraph,
    we associate a weight
of~$15$ to each white edge and a weight of~$0$ to each red edge.
Further, we associate a weight with every vertex  as follows:

\begin{center}
\begin{tabular}{|c|c|c|}  \hline
\emph{Color of vertex} & \emph{Degree in the} & \emph{Weight of vertex} \\
 & \emph{residual hg.} & \\
\hline
white & $\ge 3$ & $15$ \\
green & $2$ & $14$  \\
blue & $1$ &$11$  \\
red & --- & $0$  \\ \hline
\end{tabular}
\end{center}
\begin{center}
\textbf{Table~1.} The weights of vertices according to their color.
\end{center}

Let $\nW$, $\nGr$ and $\nB$ denote the set of
white, green and blue vertices, respectively, in the residual
hypergraph~$H$. We define the \emph{weight} of the residual
hypergraph $H$ as the sum of the weights of the vertices and edges
in $H$ and denote this weight by $\w(H)$. Thus,
\[
\w(H) = 15|\nW| + 14|\nGr| + 11|\nB| + 15\mH.
\]

We note that as the game is played, if the color status of a vertex
changes,  then the color status of a green vertex can only change to
blue or red, while the color status of a blue vertex can only change
to red. We shall prove the following key theorem. From our earlier
observations, it suffices for us to prove Theorem~\ref{3unif} in the
case of hypergraphs with no multiple edges.

\begin{thm}
If $H$ is a $3$-uniform residual hypergraph, then
$48 \dstart(H) \le \w(H)$.
\label{3unif}
\end{thm}
\begin{proof} If $\mH = 0$, then $\dstart(H) = 0$ and the desired bound is immediate. Hence we may assume that $\mH \ge 1$. We say that \Eh can \emph{achieve a $48$-target} if he can play a sequence of moves guaranteeing that on average the weight decrease resulting from each played vertex in the game is at least~$48$. In order to achieve a $48$-target, \Eh must guarantee that a sequence of moves $m_1, \ldots, m_k$ are played, starting with his first move $m_1$, and with moves alternating between \Eh and \St such that if $\w_i$ denotes the decrease in weight after move $m_i$ is played, then
\begin{equation}
\sum_{i = 1}^k \w_i \ge 48 \cdot k\,, \label{Eq1}
\end{equation}
where either $k$ is odd and the game is completed after move $m_k$
or $k$ is any even number  (in this latter case the game may or may
not be completed after move $m_k$). Each played vertex must hit at
least one edge not hit by the vertices previously chosen. Thus,
every move decreases the weight by at least~$26$, since every move
results in at least one vertex and at least one edge recolored red.

In the discussion that follows, we analyse how \Eh can achieve a $48$-target.
 First of all we note that there is a trivial situation, namely when \Eh
 can play a vertex that covers all remaining edges, and the current value
  of the residual hypergraph is at least 48.
 Then the 48-target is achieved with $k=1$.
This may happen in several cases below.
 We shall not mention it each time, we only discuss what happens otherwise.

We prove a series of claims that establish important properties
 that hold in the residual hypergraph $H$.

\begin{unnumbered}{Claim~\ref{3unif}.A}
If $\Delta(H) \ge 4$, then Edge-hitter can achieve a $48$-target.
\end{unnumbered}
\begin{proof} Suppose that $d_H(v) \ge 4$.
 If  \Eh  plays the vertex $v$ as his move $m_1$ in the residual hypergraph, this results in
 $\w_1 \ge 1 \cdot 15 + 4 \cdot 15 = 75 > 48\cdot 1$ since after the move is played,
 at least one white vertex and at least four (white) edges are recolored red.
 Then \St responds by playing her move $m_2$ which decreases the weight
 by~$\w_2\ge 11 + 15 = 26$ since her move results in at least one vertex
  and at least one edge recolored red.
 Thus,   $\w_1 + \w_2 \ge 75 + 26 = 101 > 48 \cdot 2$, and so
  Inequality~(\ref{Eq1}) is satisfied with $k = 2$.~\smallqed
\end{proof}

\medskip
By Claim~\ref{3unif}.A, we may assume that $\Delta(H) \le 3$, for otherwise Edge-hitter can achieve a $48$-target.

\begin{unnumbered}{Claim~\ref{3unif}.B}
If \Eh can play a vertex that results in a weight decrease of at least~$68$, then he can achieve a $48$-target.
\end{unnumbered}
\begin{proof}  Suppose that \Eh plays as his move $m_1$ a vertex in the residual hypergraph
  $H$ which results in  $\w_1 \ge 68 > 48\cdot 1$.
Then \St responds by playing her move
$m_2$ which decreases the weight by~$\w_2 \ge 11 + 15 + 2 \cdot 1 =
28$ since her move results in the vertex she played and at least one
edge recolored red, and at least two further vertices changing
color.  Therefore, $\w_1 + \w_2 \ge 68 + 28 = 96 = 48 \cdot 2$, and
so Inequality~(\ref{Eq1}) is satisfied with $k = 2$.~\smallqed
\end{proof}

\begin{unnumbered}{Claim~\ref{3unif}.C}
If \Eh can play a white vertex $v$ that results in at least one of
its neighbors recolored red, then \Eh can achieve a $48$-target.
\end{unnumbered}
\begin{proof}  If \Eh plays the vertex $v$ as his move $m_1$ in  $H$,
this results in $\w_1 \ge 1 \cdot 15 + 3 \cdot 15 + 11  = 71$ since
after the move is played, the vertex~$v$ and the three (white) edges
incident with $v$ are recolored red, while at least one neighbor of
$v$ is recolored red. Thus, by Claim~\ref{3unif}.B, \Eh can achieve
a $48$-target.~\smallqed
\end{proof}

\medskip
By Claim~\ref{3unif}.C, we may assume that there is no white vertex
which, when played, results in at least one of its neighbors
recolored red.

\begin{unnumbered}{Claim~\ref{3unif}.D}
If there exist two overlapping edges that contain a common white
vertex $v$, then Edge-hitter can achieve a $48$-target.
\end{unnumbered}
\begin{proof}   Since there are  no
multiple edges, the vertex $v$  has three, four or five neighbors.
\Eh plays the vertex $v$ as his move $m_1$ in the residual
hypergraph $H$. Suppose firstly that $|N(v)| = 3$. By our earlier
assumptions, no neighbor of $v$ is recolored red, implying that all
three neighbors of $v$ are white vertices in $H$ (of degree~$3$) and
are recolored blue once $v$ is played. Thus, in this case, $\w_1 \ge
1 \cdot 15 + 3 \cdot 15 + 3 \cdot 4  = 72$. Suppose secondly that
$|N(v)| = 4$. At least two neighbors of $v$ are recolored blue once
$v$ is played, implying that in this case, $\w_1 \ge 1 \cdot 15 + 3
\cdot 15 + 2 \cdot 4 + 2 \cdot 1 = 70$. Suppose thirdly that $|N(v)|
= 5$. At least one neighbor of $v$ is recolored blue once $v$ is
played, implying that in this case, $\w_1 \ge 1 \cdot 15 + 3 \cdot
15 + 1 \cdot 4 + 4 \cdot 1 = 68$. In all three cases, by
Claim~\ref{3unif}.B, \Eh can achieve a $48$-target.~\smallqed
\end{proof}

\medskip
By Claim~\ref{3unif}.D, we may assume that no white vertex belongs to the intersection of two overlapping edges. Recall that by our earlier assumptions, no white vertex which when played results in at least one of its neighbors recolored red. With these assumptions, the three edges that contain a white vertex are pairwise non-overlapping, implying that every white vertex has six neighbors. Further, these six neighbors are colored white or green.

\begin{unnumbered}{Claim~\ref{3unif}.E}
If a white vertex $v$ has a green neighbor $u$, then \Eh can achieve
a $48$-target.
\end{unnumbered}
\begin{proof}  By our earlier assumptions, $|N(v)| = 6$.
If \Eh plays the vertex $v$ as his move $m_1$ in $H$, this results
in the green neighbor $u$ recolored blue and five further neighbors
recolored. Therefore, $\w_1 \ge 1 \cdot 15 + 3 \cdot 15 + 1 \cdot 3
+ 5 \cdot 1 = 68$. Thus, by Claim~\ref{3unif}.B, \Eh can achieve a
$48$-target.~\smallqed
\end{proof}

\medskip
By Claim~\ref{3unif}.E, we may assume that every neighbor of a white
vertex  is colored white. With this assumption, every component of
$H$ is one of the following three types:

\indent
$\bullet$ \emph{Type-A:} A $3$-regular, linear hypergraph. \\
\indent
$\bullet$ \emph{Type-B:} A hypergraph with maximum degree~$2$. \\
\indent
$\bullet$ \emph{Type-C:} A hypergraph consisting of a single edge.

We remark that a Type-A component of $H$ consists entirely of white
vertices, while a type-B component consists only of green and blue
vertices, with at least one green vertex. A type-C component
consists of three blue vertices. Since a type-B component contains
only green and blue vertices, a move played in such a component
decreases the weight by at least~$11 + 15 + 2 \cdot 3 = 32$, since
at least one vertex and one edge is recolored red, and at least two
further vertices are recolored.  A move played in a Type-C component
decreases the weight by $3 \cdot 11 + 1 \cdot 15 = 48$, since three
blue vertices and one edge are recolored red. We state this formally
as follows.

\begin{unnumbered}{Claim~\ref{3unif}.F}
A move played in a type-B component decreases the weight by at
least~$32$,  while a move played in a Type-C component decreases the
weight by~$48$.
\end{unnumbered}

\begin{unnumbered}{Claim~\ref{3unif}.G}
If $H$ contains a white vertex, then \Eh can achieve a $48$-target.
\end{unnumbered}
\begin{proof}  Suppose that $H$ contains a white vertex, $v$,
that belongs to a component $F$. We note that $F$ is a type-A
component.   \Eh plays the vertex $v$ as his move $m_1$ in the
residual hypergraph $H$, which results in  $\w_1 \ge 1 \cdot 15 + 3
\cdot 15 + 6 \cdot 1 = 66 > 48\cdot 1$ since after the move is
played, the vertex~$v$ and three edges are recolored red, while all
six neighbors of $v$ are recolored green.
 Then \St responds by playing her move $m_2$. We note
that $F - v$ is a linear (possibly disconnected) hypergraph that
contains six green vertices with all other vertices colored white.
If \St plays her move $m_2$ in $F - v$ or in a Type-A component,
then $\w_2 \ge 14 + 2 \cdot 15 + 4 \cdot 1 = 48$ since her played
vertex (colored either white or green) and at least two edges are
recolored red, while at least four further vertices are recolored.
If \St plays her move $m_2$ in a Type-B component, then, by
Claim~\ref{3unif}.F, $\w_2 \ge 32$. If \St plays her move $m_2$ in a
Type-C component, then, by Claim~\ref{3unif}.F, $\w_2 \ge 48$. In
all cases, $\w_2 \ge 32$. Therefore, $\w_1 + \w_2 \ge 66 + 32 = 98 >
48 \cdot 2$, and so Inequality~(\ref{Eq1}) is satisfied with $k =
2$.~\smallqed
\end{proof}

\medskip
By Claim~\ref{3unif}.G, we may assume that every vertex is colored
green or blue; that is, every component of $H$ is of Type-B or
Type-C. We have seen in Claim~\ref{3unif}.F that in this situation
\St can never make a decrease smaller than $32$. Thus, we obtain:

\begin{unnumbered}{Claim~\ref{3unif}.H}
If \Eh can play a vertex that results in a weight decrease of at least~$64$, then he can achieve a $48$-target.
\end{unnumbered}

\begin{unnumbered}{Claim~\ref{3unif}.I}
If $H$ contains two overlapping edges, then \Eh can achieve a $48$-target.
\end{unnumbered}
\begin{proof}
 Let $e$ and $f$ be two overlapping edges, with $e \cap f = \{v_1,v_2\}$.
 \Eh plays the vertex $v_1$ as his move
$m_1$ in the residual hypergraph $H$, which results in $\w_1 \ge 2
\cdot 14 + 2 \cdot 15 + 2 \cdot 3 = 64$ since after the move is
played, both vertices $v_1$ and $v_2$ (currently colored green) and
two edges are recolored red, while at least two further vertices are
recolored (from green to blue, or from blue to red). Thus, by
Claim~\ref{3unif}.H, \Eh can achieve a $48$-target.~\smallqed
\end{proof}

\medskip
By Claim~\ref{3unif}.I, we may assume that every Type-B component is linear. Thus, $H$ is a linear hypergraph.

\begin{unnumbered}{Claim~\ref{3unif}.J}
If $H$ contains a green vertex $v$ with a blue neighbor $u$, then
\Eh can achieve a $48$-target.
\end{unnumbered}
\begin{proof}   Since $H$ is linear, we note that $|N(v)| = 4$.
Playing the vertex~$v$ results in $\w_1 \ge 1 \cdot 14 + 2 \cdot 15
+ 11 + 3 \cdot 3 = 64$, since after the move is played, the
vertex~$v$ and two edges are recolored red, and $u$ is recolored
red. Thus, by Claim~\ref{3unif}.H, \Eh can achieve a
$48$-target.~\smallqed
\end{proof}

\medskip
By Claim~\ref{3unif}.J, we may assume that each component of $H$ is
either a $2$-regular linear hypergraph (consisting entirely of green
vertices) or an isolated edge (consisting of three blue vertices).
Playing a vertex from an isolated edge decreases the weight by $48$,
therefore we obtain:

\begin{unnumbered}{Claim~\ref{3unif}.K}
If every component in the residual hypergraph is an isolated edge, then \Eh can achieve a $48$-target.
\end{unnumbered}

\medskip
By Claim~\ref{3unif}.K, we may assume that at least one component,
say $F$, of $H$ is a $2$-regular, linear hypergraph. Edge-hitter now
plays in such a way as to restrict his moves to vertices in $V(F)$,
independently of Staller's responses to his moves, as long as a
green vertex in $V(F)$ exists. Further, among all green vertices in
$V(F)$ at each stage of the game, \Eh selects a green vertex with as
many blue neighbors as possible.

Suppose that a total of $\ell$ green vertices in $V(F)$ are played
by Edge-hitter.  The first move, $m_1$, of \Eh results in $\w_1 = 1
\cdot 14 + 2 \cdot 15 + 4 \cdot 3 = 56$. Thereafter, each subsequent
move $m_{2i+1}$ of Edge-hitter, where $i \in \{1,\ldots,\ell - 1\}$,
results in $\w_{2i+1} \ge 1 \cdot 14 + 2 \cdot 15 + 1 \cdot 11 + 3
\cdot 3 = 64$, since the subsequent (green) vertices played by \Eh
in $V(F)$ can all be chosen to have at least one blue neighbor. By
Claim~\ref{3unif}.F,   each of Staller's moves $m_{2i}$, where $i
\in \{1,\ldots,\ell-1\}$, result in $\w_{2i} \ge 32$. If the game is
complete after Edge-hitter's $\ell$th move, then

\[
\begin{array}{lcl}
\displaystyle{ \sum_{i=1}^{2\ell-1} \w_i} & =
& \displaystyle{ \sum_{i=0}^{\ell-1} \w_{2i+1} + \sum_{i=1}^{\ell - 1} \w_{2i} } \1 \\
& \ge & \displaystyle{ 56 + 64(\ell - 1)  + 32(\ell - 1) }  \1 \\
& > & 48 \cdot (2\ell - 1).
\end{array}
\]

\noindent Thus, Inequality~(\ref{Eq1}) is satisfied with $k = 2\ell
- 1$  and the game is completed after move $m_k$. Hence, we may
assume that the game is not complete after Edge-hitter's $\ell$th
move. We show next that Inequality~(\ref{Eq1}) is satisfied with $k
= 2\ell$. We consider the sequence of $\ell$ vertices played by \St
in response to Edge-hitter's $\ell$ moves. As observed earlier,
every move played by Staller decreases the weight by at least~$32$.

\begin{unnumbered}{Claim~\ref{3unif}.L}
If one of Staller's moves in response to Edge-hitter's $\ell$ moves is not a blue vertex in $V(F)$ with two green neighbors, then \Eh can achieve a $48$-target.
\end{unnumbered}
\begin{proof}  Suppose that \St plays a move that is not a blue vertex
in $V(F)$ with two green neighbors. We consider the four possible
moves of Staller. If at least one of the $\ell$ vertices played by
\St does not belong to $V(F)$, then her first such played vertex
either belongs to a component of $H$, different from $F$, that is a
$2$-regular, linear hypergraph or belongs to a component of $H$ that
is an isolated edge. In the former case, her move decreases the
weight by~$56$, while in the latter case, her move decreases the
weight by~$48$. If \St plays a green vertex in $V(F)$, then her move
decreases the weight by at least~$56$. If \St plays a blue vertex in
$V(F)$ that has at least one blue neighbor, then her move decreases
the weight by at least~$2 \cdot 11 + 1 \cdot 15 + 3 = 40$. In all
four cases, Staller's move decreases the weight by at least~$40$,
while the other $\ell - 1$ moves played by her each decrease the
weight by at least~$32$, implying that

\[
\begin{array}{lcl}
\displaystyle{ \sum_{i=1}^{2\ell} \w_i} & =
& \displaystyle{ \sum_{i=0}^{\ell-1} \w_{2i+1} + \sum_{i=1}^{\ell} \w_{2i} } \1 \\
& \ge & \displaystyle{ \left( 56 + 64(\ell - 1) \right)   + \left( 40 + 32(\ell - 1) \right) }  \1 \\
& = & 48 \cdot (2\ell).
\end{array}
\]
Thus, Inequality~(\ref{Eq1}) is satisfied with $k = 2\ell$.~\smallqed
\end{proof}

\medskip
By Claim~\ref{3unif}.L, we may assume that each move played by \St in response to Edge-hitter's $\ell$ moves is a blue vertex in $V(F)$ with two green neighbors. Thus, each of the $\ell$ moves of \St decreases the weight by exactly~$32$.

\begin{unnumbered}{Claim~\ref{3unif}.M}
At least one move played by \Eh is a green vertex in $V(F)$ with at least two blue neighbors.
\end{unnumbered}
\begin{proof}  Suppose that none of the $\ell$ moves played by \Eh is a green
vertex in $V(F)$  with at least two blue neighbors. Then, every move
of Edge-hitter, except for his first move, plays a green vertex with
exactly one blue neighbor and three green neighbors. Thus, the first
move of \Eh recolors exactly five green vertices, while each of the
subsequent $\ell - 1$ moves of \Eh  recolors exactly four green
vertices. By our earlier assumption, each move played by \St in
response to Edge-hitter's $\ell$ moves is a blue vertex in $V(F)$
with two green neighbors. Thus, each move of \St recolors exactly
two green vertices. Therefore,  $|V(F)| = 5 + 4(\ell - 1) + 2\ell =
6\ell + 1$. However, $F$ is a $2$-regular, $3$-uniform hypergraph,
and so $m(F) = \frac{2}{3}|V(F)|$, implying that $|V(F)|$ must be
divisible by~$3$, a contradiction.~\smallqed
\end{proof}

\medskip
By Claim~\ref{3unif}.M, at least one of the $\ell$ moves played by \Eh is a green vertex in $V(F)$ with at least two blue neighbors. Such a move decreases the weight by at least $1 \cdot 14 + 2 \cdot 11 + 2 \cdot 15 + 2 \cdot 3 = 72$, implying that

\[
\begin{array}{lcl}
\displaystyle{ \sum_{i=1}^{2\ell} \w_i} & =
& \displaystyle{ \sum_{i=0}^{\ell-1} \w_{2i+1} + \sum_{i=1}^{\ell} \w_{2i} } \1 \\
& \ge & \displaystyle{ (56 + 72 + 64(\ell - 2) ) + 32\ell }  \1 \\
& = & 48 \cdot (2\ell).
\end{array}
\]
Thus, Inequality~(\ref{Eq1}) is satisfied with $k = 2\ell$. This completes the proof of Theorem~\ref{3unif}.~\qed
\end{proof}

\medskip
For a hypergraph $H$, we let $n_{\ge 3}(H)$ denote the number of vertices of degree at least~$3$ in $H$. Further, we let $n_2(H)$ and $n_1(H)$ denote the number of vertices of degree~$2$ and~$1$, respectively, in $H$. We observe that Theorem~\ref{3unif} can be restated as follows.

\begin{thm}
\label{3uniformA}
If $H$ is a $3$-uniform hypergraph, then
\[
48 \dstart(H) \le 15n_{\ge 3}(H) + 14n_2(H) + 11n_1(H) + 15\mH.
\]
\end{thm}

Since the right side is at most $15 \nH + 15\mH$,
Theorem~\ref{3uniform} is an immediate consequence of
Theorem~\ref{3unif} and Theorem~\ref{3uniformA}. Recall the
statement of Theorem~\ref{3uniform}.

\noindent \textbf{Theorem~\ref{3uniform}}. \emph{If $H$ is a $3$-uniform hypergraph, then  $\dstart(H) \le \frac{5}{16}(\nH + \mH)$.
}

As a further consequence of Theorem~\ref{3uniformA}, we have the following upper bound on the game transversal number of a $3$-uniform hypergraph with maximum degree at most~$2$.

\begin{cor}
If $H$ is a $3$-uniform hypergraph and $\Delta(H) \le 2$, then the following holds. \1 \\
\indent {\rm (a)} $\dstart(H) \le \frac{3}{10}(\nH+\mH)$. \1 \\
\indent {\rm (b)}  $\dstart(H) \le \frac{1}{2}\nH$. \1 \\
\indent {\rm (c)} $\dstart(H) \le \frac{3}{4}\mH$ if\/ $H$ is\/ $2$-regular.
\label{c:Delat2}
\end{cor}
\begin{proof} If $H$ is a $3$-uniform hypergraph and $\Delta(H) \le 2$, then $\mH \le \frac{2}{3} \nH$, implying, by  Theorem~\ref{3uniformA},
that $\dstart(H) \le \frac{1}{48}(14\nH + 15\mH) \le \frac{3}{10}(\nH+\mH) \le \frac{1}{2}\nH$,
  which is equal to $\frac{3}{4} \mH$ whenever $H$ is 3-uniform and 2-regular.~\qed
\end{proof}

A small example attaining equation in all of Theorem~\ref{3uniformA}
and Corollary~\ref{c:Delat2} (a)--(c) is shown in
Figure~\ref{fig:2}.

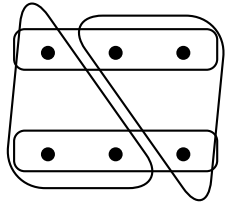
\begin{figure}[t]
\begin{center}

\begin{tikzpicture}[scale=0.9,style=thick,x=1cm,y=1cm]
\def\vr{2.5pt} 
\path (0,0) coordinate (x1); \path (1,0) coordinate (x2); \path
(2,0) coordinate (x3); \path (0,1.5) coordinate (y1); \path (1,1.5)
coordinate (y2); \path (2,1.5) coordinate (y3);
%
\draw (x1) [fill=black] circle (\vr); \draw (x2) [fill=black] circle
(\vr); \draw (x3) [fill=black] circle (\vr); \draw (y1) [fill=black]
circle (\vr); \draw (y2) [fill=black] circle (\vr); \draw (y3)
[fill=black] circle (\vr); \draw [rounded corners] (-0.5,-0.25)
rectangle (2.5,0.35) node [black,right] {}; \draw [rounded corners] (-0.5,1.25) rectangle (2.5,1.85) node [black,right] {}; %
\draw [rounded corners=5.5mm] (-0.65,-0.5)--(-0.35,
2.55)--(1.8,-0.5)--cycle;
 \draw [rounded corners=5.5mm]
(0.2,2.05)--(2.65, 2.05)--(2.35,-1)--cycle;

\end{tikzpicture}
\end{center}
\vskip -0.25cm \caption{The $3$-uniform hypergraph $H$ with $\nH=6$,
$\mH=4$ and $\dstart(H)=3$.} \label{fig:2}
\end{figure}

\medskip
For the Staller-start game, we have the following consequence of Theorem~\ref{3uniform}.

\begin{cor}
If $H$ is a $3$-uniform hypergraph, then $\sstart(H) \le \frac{1}{16}(5\nH + 5\mH + 6)$. \label{cor-3uniform}
\end{cor}
\begin{proof} The first move of \St decreases $\nH + \mH$  by at least~$2$, since at least
 one vertex and one edge are deleted by her move. Let $H'$ denote the resulting residual hypergraph. Then $\nHp + \mHp \le \nH + \mH - 2$. By Theorem~\ref{3uniform},
\[
\begin{array}{lcl}
\sstart(H) & = & 1 + \dstart(H') \1 \\
&  \le & 1 + \frac{5}{16}(\nHp + \mHp) \1 \\
&  \le & 1 + \frac{5}{16}(\nH + \mH - 2) \1 \\
& = & \frac{1}{16}( 5\nH + 5\mH + 6). \hspace*{0.5cm} \Box
\end{array}
\]
\end{proof}

\subsection{Proof of Theorem~\ref{4uniform}}
\label{S:4uniform}

 In this section, we prove
Theorem~\ref{4uniform}. Again, we consider colored hypergraphs,
where each edge and vertex is associated with a color.
 The colors of edges and vertices
  may change as the set $D$ of chosen vertices is
extended during the game.

An edge is colored \emph{white} if it is not covered by a vertex of
$D$, and is colored \emph{red} otherwise.  From the partially
covered hypergraph  red edges and isolated vertices are deleted.
This way we obtain the residual hypergraph.

Each vertex of  the  hypergraph is associated with one from the
following five colors:  white, yellow, green, blue, and red. This
coloring reflects to the degree of the vertex in the residual
hypergraph; that is, the number of white edges incident to it.
\begin{itemize}
\item A vertex is colored \emph{white} if it has degree at least~$4$.
\item A vertex is colored \emph{yellow} if it has degree~$3$.
\item A vertex is colored \emph{green} if it has degree~$2$.
\item A vertex is colored \emph{blue} if it has degree~$1$.
\item A vertex is colored \emph{red} if it is not incident with any
white edges or equivalently, if it is deleted from the residual
hypergraph.
\end{itemize}

We now define the parameter $\Delta^*(H)$ of the residual hypergraph
$H$  of the game as follows. If Edge-hitter is the next player to
make a move on $H$, then $\Delta^*(H)$ is the maximum degree,
$\Delta(H)$, of $H$. Otherwise, if Staller is the next player to
make a move on $H$, then $\Delta^*(H)$ denotes the maximum degree of
the residual hypergraph before Edge-hitter made his previous move.
We associate a weight with every vertex in the residual hypergraph
$H$ that depends on $\Delta^*(H)$ and on the color of the vertex in
$H$.

\begin{center}
\begin{tabular}{|c|c|c|c|c|c|}  \hline
\emph{Color of vertex} & \emph{Degree in $H$} & \multicolumn{4}{c}{\emph{Weight of vertex}} \vline \\
\cline{3-6}
   & &$\Delta^*(H) \ge 5$ & $\Delta^*(H) =4$ & $\Delta^*(H) =3$ & $\Delta^*(H)
    \le 2$ \\ \hline
white & $\ge 4$ &$852$ & $852$ & -- & -- \\
yellow & $3$ & $852$ & $845$ & $845$ & -- \\
green & $2$ & $852$  & $838$ & $750$ & $750$ \\
blue & $1$ & $852$  & $831$ & $655$ & $543$ \\
red & --- & $0$ & $0$ & $0$ & $0$  \\ \hline
\end{tabular}
\end{center}
\begin{center}
\textbf{Table~2.} The weights of vertices according to their color
and $\Delta^*(H)$.
\end{center}

Further, the weight of an edge is $852$ if it is white, and $0$ if it is red. We shall prove the following key theorem. From our earlier observations, it suffices for us to prove Theorem~\ref{4unif} in the case of hypergraphs with no multiple edges.

\begin{thm}
If $H$ is a $4$-uniform residual hypergraph, then $3024 \dstart(H) \le \w(H)$.
 \label{4unif}
\end{thm}
\begin{proof} If $\mH = 0$, then $\dstart(H) = 0$ and the desired bound is immediate. Hence we may assume that $\mH \ge 1$. We say that \Eh can \emph{achieve a $3024$-target} if he can play a sequence of moves guaranteeing that on average the weight decrease resulting from each played vertex in the game is at least~$3024$. In order to achieve a $3024$-target, \Eh must guarantee that a sequence of moves $m_1, \ldots, m_k$ are played, starting with his first move $m_1$, and with moves alternating between \Eh and \St such that if $\w_i$ denotes the decrease in weight after move $m_i$ is played, then
\begin{equation}
\sum_{i = 1}^k \w_i \ge 3024 \cdot k\,, \label{Eq2}
\end{equation}
where either $k$ is odd and the game is completed after move $m_k$
or  $k$ is any even number (in this latter case the game may or may
not be completed after move $m_k$).

We will analyze how  \Eh can achieve a $3024$-target. Similarly to
the proof of Theorem \ref{3unif}, there is a trivial situation,
 namely when \Eh\ can play a vertex that covers all remaining edges,
  and the current value of the residual hypergraph is at least 3024.
 Then the 3024-target is achieved with $k=1$.
This may happen in several cases below.
 We shall not mention it each time, we only discuss what happens otherwise.

 We prove a series of claims that establish important properties that hold in
the residual hypergraph $H$.

\begin{unnumbered}{Claim~\ref{4unif}.A}
If $\Delta(H) \ge 5$, then Edge-hitter can achieve a $3024$-target.
\end{unnumbered}
\begin{proof}
   Let $v$ be a (white) vertex of
degree at least~$5$ in $H$. If  \Eh plays the vertex $v$, this
results in at least five edges  recolored red. Moreover, the white
vertex $v$ is recolored red. Hence, since $\Delta^*(H) = \Delta(H)
\ge 5$ immediately before Edge-hitter plays~$v$,  $\w_1 \ge 5\cdot
852 + 1 \cdot 852 = 5112 > 3024 \cdot 1$.
%
   Similarly, $\Delta^*(H) \ge 5$ before Staller makes her next move.
Thus, Staller's move $m_2$ decreases the weight by $\w_2 \ge 852 +
852 = 1704$ and we have $\w_1 + \w_2 \ge 5112 + 1704 = 6816 > 3024
\cdot 2$. Therefore, Inequality~(\ref{Eq2}) is satisfied with $k =
2$.~\smallqed
\end{proof}

By Claim~\ref{4unif}.A, we may assume $\Delta(H) \le 4$, for otherwise Edge-hitter can achieve a $3024$-target.

\begin{unnumbered}{Claim~\ref{4unif}.B}
If $\Delta(H) =4$, then Edge-hitter can achieve a $3024$-target.
\end{unnumbered}
\begin{proof}
Suppose that \Eh plays a  (white) vertex $v$  of degree~$4$ in $H$.
We note that immediately before Edge-hitter plays~$v$, $\Delta^*(H)
= \Delta(H) = 4$. If the degree of a vertex $x$ is decreased by
$\ell$ after the vertex~$v$ is played, then the weight $\w(x)$ of
$x$ decreases by at least $7\ell$. When \Eh plays~$v$, four white
edges and the white vertex $v$ are recolored red. Further,
 $\dsum{v}$
 decreases by exactly~$12$.
Therefore, this move results in~$\w_1 \ge 4\cdot 852+ 1 \cdot 852 +
12 \cdot 7 = 4344 > 3024 \cdot 1$.
 In the next turn Staller plays a vertex, $u$ say.
We note that immediately before Staller makes her move, $\Delta^*(H)
= 4$. Staller's move results in the vertex $u$ recolored red, and so
the weight $\w(u)$ of $u$ decreases by at least $831$. Her move also
results in at least one edge recolored red. Further,
 $\dsum{u}$
 decreases by at least~$3$. Thus, Staller's
move $m_2$ decreases the weight by at least~$831 + 852 + 3 \cdot 7 =
1704$. Therefore, $\w_1 + \w_2 \ge 4344 + 1704 = 6048 = 3024 \cdot
2$, and so Inequality~(\ref{Eq2}) is satisfied with $k =
2$.~\smallqed
\end{proof}

By Claim~\ref{4unif}.B, we may assume $\Delta(H) \le 3$, for otherwise Edge-hitter can achieve a $3024$-target.

\begin{unnumbered}{Claim~\ref{4unif}.C}
If $\Delta(H) =3$, then Edge-hitter can achieve a $3024$-target.
\end{unnumbered}
\begin{proof}
Suppose that $\Delta(H) = 3$ and \Eh plays a (yellow) vertex $v$ of
degree~$3$ in $H$. Then, before Edge-hitter and Staller make  their
next moves, $\Delta^*(H) = 3$.  If the degree of a vertex $x$ is
decreased by $\ell$,   the weight  of $x$ decreases by at least
$95\ell$. When \Eh plays~$v$, three white edges and the yellow
vertex $v$ are recolored red. Further,
  $\dsum{v}$
 decreases by exactly~$9$. Therefore, playing the
vertex~$v$ results in~$\w_1 \ge 3 \cdot 852 + 1 \cdot 845 + 95 \cdot
9 = 4256 > 3024 \cdot 1$.
 In the next turn Staller plays a vertex, $u$ say. Staller's move
results in the vertex $u$ recolored red, and so the weight of $u$
decreases by at least $655$. Her move also results in at least one
edge recolored red. Further,
 $\dsum{u}$
decreases by at least~$3$. Thus, Staller's move $m_2$ decreases the
weight by at least~$655 + 852 + 3 \cdot 95 = 1792$. Therefore, $\w_1
+ \w_2 \ge 4256 + 1792 = 6048 = 3024 \cdot 2$, and so
Inequality~(\ref{Eq2}) is satisfied with $k = 2$.~\smallqed
\end{proof}

By Claim~\ref{4unif}.C, we may assume $\Delta(H) \le 2$, for otherwise Edge-hitter can achieve a $3024$-target.

\begin{unnumbered}{Claim~\ref{4unif}.D}
If $\Delta(H) = 2$, and $H$ contains two overlapping edges or $H$ contains a green vertex with a blue neighbor, then \Eh can achieve a $3024$-target.
\end{unnumbered}
\begin{proof}
Suppose that $\Delta(H) = 2$.  Then, before Edge-hitter and Staller
make their next moves, $\Delta^*(H) = 2$. If the degree of a vertex
$x$ is decreased by $\ell$ after the vertex~$v$ is played, then the
weight  of $x$ decreases by at least $207 \ell$.

Suppose firstly that $H$ contains two overlapping edges;
 say $u$ and $v$ are two vertices common to them.
 \Eh now plays the
  vertex $v$. The two edges incident with $v$ are recolored
red, as are both green vertices $u$ and $v$. Further,
  $\sum_{w\in N(v)\setminus\{u\}} d_H(w)$
 decreases by~$4$.
Therefore, playing the vertex~$v$ results in~$\w_1 \ge 2 \cdot 852 +
2 \cdot 750 + 4 \cdot 207 = 4032 > 3024 \cdot 1$.
 In the next turn Staller plays a vertex,  $w$ say.
Staller's move results in the vertex $w$ and at least one edge
recolored red. The degree sum of the neighbors of $w$ decreases
 by at least~$3$. Thus, Staller's move $m_2$ decreases the weight
by~$\w_2 \ge 543 + 852 + 3 \cdot 207 = 2016$. Therefore, $\w_1 +
\w_2 \ge 4032 + 2016 = 6048 = 3024 \cdot 2$, and so \Eh achieves a
$3024$-target.

Suppose secondly that $H$ contains a green vertex, $v$, having a
blue neighbor, $w$.  \Eh now plays the vertex $v$ which results in
$\w_1 \ge 750 + 543 + 2 \cdot 852 + 5 \cdot 207 = 4032 > 3024 \cdot
1$, since after the move is played, the green vertex~$v$ and its
blue neighbor $w$ are recolored red, and the two white edges
incident with $v$ are recolored red. Further,
 the degree sum of the neighbors of $v$ different from~$w$
 decreases by~$5$. Therefore,
analogously as before, Inequality~(\ref{Eq2}) is satisfied with $k =
1$ or $k=2$.~\smallqed
\end{proof}

By Claim~\ref{4unif}.D, we may assume that every component of the
 residual hypergraph $H$ is either a $2$-regular, linear hypergraph
  or an isolated edge (consisting of
 four  blue vertices),
   for otherwise Edge-hitter can achieve a $3024$-target.

\begin{unnumbered}{Claim~\ref{4unif}.E}
If there is an isolated edge in $H$, then \Eh can achieve a $3024$-target.
\end{unnumbered}
\begin{proof}
 If the assumption holds, \Eh can play a vertex from an isolated edge, what
 results in
$\w_1 = 4 \cdot 543 + 852 = 3024 = 3024 \cdot 1$.
 In the next turn Staller plays either a vertex from an
isolated edge and $\w_2 = 3024$, or a vertex from a $2$-regular,
linear component. In the latter case, two white edges and the played
green vertex are recolored red, and  further six green vertices are
recolored blue, implying that $\w_2 = 750 + 2 \cdot 852 + 6 \cdot
207 = 3696$. In both cases, $\w_1 + \w_2 \ge 6048 = 3024 \cdot 2$,
and so Inequality~(\ref{Eq2}) is satisfied with $k = 2$.~\smallqed
\end{proof}


 By Claim~\ref{4unif}.E, we may assume that every component of
the residual hypergraph $H$ is a $2$-regular, linear hypergraph, for
otherwise Edge-hitter can achieve a $3024$-target.
 \Eh now selects a component $C$ of $H$, and will play inside $C$
  as long as a green vertex in $V(C)$ exists,
independently of Staller's responses to his moves.
 More explicitly, among all green vertices in
$V(C)$ at each stage of the game, \Eh plays a green vertex with as
many blue neighbors as possible. We note that subsequent to his
first move, as long as a green vertex in $V(C)$ exists, \Eh can play
a green vertex having at least one blue neighbor.

Suppose that a total of $s$ green vertices in $V(C)$ are played by Edge-hitter. The first move, $m_1$, of \Eh results in $\w_1 = 750 + 2 \cdot 852 + 6 \cdot 207 = 3696$. Thereafter, each subsequent move $m_{2j+1}$ of Edge-hitter, where $j \in \{1,\ldots,s - 1\}$, results in $\w_{2j+1} \ge 750 + 2 \cdot 852 + 543 + 5 \cdot 207 = 4032$.
Every move played by Staller decreases the weight by at least~$543 + 852 + 3 \cdot 207 = 2016$, since with each of her moves at least one edge and a vertex are recolored red,
 and the degree sum of the remaining vertices is decreased by at least~$3$. In particular, each of Staller's moves $m_{2j}$, where $j \in \{1,\ldots,s-1\}$, results in $\w_{2j} \ge 2016$. If the game is complete after Edge-hitter's $s$th move, then

\[
\begin{array}{lcl}
\displaystyle{ \sum_{j=1}^{2s-1} \w_j} & =
& \displaystyle{ \sum_{j=0}^{s-1} \w_{2j+1} + \sum_{j=1}^{s - 1} \w_{2j} } \1 \\
& \ge & \displaystyle{ 3696 + 4032(s - 1)  + 2016(s - 1) }  \1 \\
& = & 3696 +  3024 \cdot 2(s - 1) \1 \\
& > & 3024 \cdot (2s - 1).
\end{array}
\]

\noindent
Thus, Inequality~(\ref{Eq2}) is satisfied with $k = 2s - 1$ and the game is completed after move $m_k$. Hence, we may assume that the game is not complete after Edge-hitter's $s$th move.

\begin{unnumbered}{Claim~\ref{4unif}.F}
If there are no green vertices in $V(C)$ after Edge-hitter's $s$th move, then \Eh can achieve a $3024$-target.
\end{unnumbered}
\begin{proof} Suppose that after Edge-hitter's $s$th move, which is the $(2s-1)$st turn in the game, all vertices in $V(C)$ in the resulting residual hypergraph are colored blue. Let $v$ be the vertex played by \Eh in his $s$th move, and let $e_1$ and $e_2$ be the two edges incident with~$v$. We show that $v$ had at least two blue neighbors. Suppose, to the contrary, that $e_1 \cup e_2$ contains only one blue vertex, say $u \in e_1$, before \Eh plays the vertex~$v$. We now consider a vertex, $u'$, from $e_1$ that is different from $u$ and $v$. We note that $u'$ is a green vertex. Let $e'$ be the edge incident with $u'$ that is different from $e$. After Edge-hitter's $s$th move, all remaining vertices are colored blue. In particular, the three vertices in $e' \setminus \{u'\}$ are all colored blue.
 Moreover, by the linearity of $C$, at most one of them belongs to $e_1 \cup e_2$. Hence, before Edge-hitter's $s$th move, the green vertex $u'$ had at least two blue neighbors
 (in fact at least three together with $u$, but we don't need this now),
 which contradicts the rule that \Eh plays a green vertex with the largest number of blue neighbors. Therefore, the vertex $v$ had at least two blue neighbors. Thus, $\w_{2s-1} \ge 750 + 2 \cdot 852 + 2 \cdot 543 + 4 \cdot 207 = 4368$. Staller's $s$th move results in $\w_{2s} \ge 2016$. Hence,
\[
\begin{array}{lcl}
\displaystyle{ \sum_{j=1}^{2s} \w_j} & =
& \displaystyle{ \sum_{j=0}^{s-1} \w_{2j+1} + \sum_{j=1}^{s} \w_{2j} } \1 \\
& \ge & \displaystyle{ ( 3696 + 4032(s - 2) + 4368)
 + 2016s }  \1 \\
& = & 3024 \cdot (2s).
\end{array}
\]
Thus, Inequality~(\ref{Eq2}) is satisfied with $k = 2s$.~\smallqed
\end{proof}

\medskip
By Claim~\ref{4unif}.F, we may assume that there is at least one green vertex in $V(C)$ after Edge-hitter's $s$th move, but after Staller's $s$th move there are no green vertices in $V(C)$, for otherwise Edge-hitter can achieve a $3024$-target. Necessarily, Staller's $s$th move plays a vertex from $V(C)$.

\begin{unnumbered}{Claim~\ref{4unif}.G}
If Staller's $s$th move does not play a blue vertex with three green neighbors, then \Eh can achieve a $3024$-target.
\end{unnumbered}
\begin{proof}
Under the given assumptions,
 Staller's $s$th move either plays a green vertex, in which case $\w_{2s} \ge 750 + 2 \cdot 852 + 6 \cdot 207 = 3696$, or plays a blue vertex having a blue neighbor, in which case $\w_{2s} \ge 2 \cdot 543 + 852 + 2 \cdot 207 = 2352$. In both cases, $\w_{2s} \ge 2352$, implying that
\[
\begin{array}{lcl}
\displaystyle{ \sum_{j=1}^{2s} \w_j} & =
& \displaystyle{ \sum_{j=0}^{s-1} \w_{2j+1} + \sum_{j=1}^{s} \w_{2j} } \1 \\
& \ge & \displaystyle{ (3696 +  3024 \cdot 2(s - 1))
 + 2352 }  \1 \\
& = & 3024 \cdot (2s).
\end{array}
\]
Thus, Inequality~(\ref{Eq2}) is satisfied with $k = 2s$.~\smallqed
\end{proof}

By Claim~\ref{4unif}.G, we may assume that Staller's $s$th move plays a blue vertex, $v$, with three green neighbors, say $u_1$, $u_2$, and $u_3$. Let $e$ be the edge incident with~$v$, and let $e_i$ be the edge incident with $u_i$ that is different from $e$ for $i=1,2,3$.
 Since no green vertices remain after Staller plays the vertex~$v$, we note that the three edges $e_1$, $e_2$ and $e_3$ are vertex-disjoint. Further, immediately before \St plays her $s$th move, the vertex set $S = e_1 \cup e_2 \cup e_3 \cup \{v\}$ contains exactly ten blue vertices and three green vertices. In the $(2s-1)$st turn, \Eh played as his $s$th move a green vertex of degree~$2$ which is not  contained in $S$, and therefore his move recolored at most six vertices in $S$ from green to blue. Thus, before the
$(2s-1)$st turn, the set $S$ contained at least four blue vertices and, by the Pigeonhole Principle, at least one of the vertices $u_1$, $u_2$, $u_3$ had at least two blue neighbors. According to Edge-hitter's rule, on the $(2s-1)$st turn when he played his $s$th move, he therefore selected a green vertex with at least two blue neighbors, implying that $\w_{2s-1} \ge 750 + 2 \cdot 852 + 2 \cdot 543 + 4 \cdot 207 = 4368$. Staller's $s$th move results in $\w_{2s} \ge 2016$. Hence,
\[
\begin{array}{lcl}
\displaystyle{ \sum_{j=1}^{2s} \w_j} & =
& \displaystyle{ \sum_{j=0}^{s-1} \w_{2j+1} + \sum_{j=1}^{s} \w_{2j} } \1 \\
& \ge & \displaystyle{ ( 3696 + 4032(s - 2) + 4368)
 + 2016s }  \1 \\
& = & 3024 \cdot (2s).
\end{array}
\]
Thus, Inequality~(\ref{Eq1}) is satisfied with $k = 2s$. This completes the proof of Theorem~\ref{4unif}.~\qed
\end{proof}

As an immediate consequence of Theorem~\ref{4unif}, we have that if $H$ is a $4$-uniform hypergraph, then
$3024 \dstart(H) \le 852\nH + 852\mH$, and so Theorem~\ref{4uniform} is an immediate consequence of Theorem~\ref{4unif}. Recall the statement of Theorem~\ref{4uniform}.

\noindent \textbf{Theorem~\ref{4uniform}}. \emph{If $H$ is a $4$-uniform hypergraph, then  $\tau_g(H) \le
\frac{71}{252}(\nH + \mH)$.
}

From the proof of Theorem~\ref{4unif} we also derive:

\begin{cor}
If $H$ is a $4$-uniform hypergraph and $\Delta(H) \le 2$, then
  $\dstart(H) \le \frac{7}{18}\nH$, moreover
  $\dstart(H) \le \frac{7}{9}\mH$ if $H$ is 2-regular.
\label{c:4-unif-Delta2}
\end{cor}
\begin{proof} If $H$ is $4$-uniform and has $\Delta(H) \le 2$, then
 $\mH \le \frac{1}{2} \nH$.
Recall that, under the assumption $\Delta^*(H) \le 2$,
 the weight of a green vertex is 750,
 and that of a white edge is 852.
Thus, by Theorem~\ref{4unif}, we obtain:
 $$
   3024\dstart(H) \le \textrm{w}(H) \le 750\nH + 852\mH
     \le 750\nH + 426\nH = 1176\nH .
 $$
This means $\dstart(H)\le \frac{7}{18}\nH$, which is
 precisely $\frac{7}{9}\mH$ if $H$ is 2-regular and 4-uniform.~\qed
\end{proof}

\medskip
For the Staller-start game, Theorem~\ref{4uniform} has
 the following further consequence.

\begin{cor}
If $H$ is a $4$-uniform hypergraph, then $\sstart(H) \le \frac{1}{252}(71 \nH + 71 \mH + 110)$.
\end{cor}

\begin{proof} The first move of \St decreases $\nH + \mH$  by at least~$2$, since at least
 one vertex and one edge are deleted by her move. Let $H'$ denote the resulting residual hypergraph. Then $\nHp + \mHp \le \nH + \mH - 2$. By Theorem~\ref{4uniform},
\[
\begin{array}{lcl}
\sstart(H) & = & 1 + \dstart(H') \1 \\
&  \le & 1 + \frac{71}{252}(\nHp + \mHp) \1 \\
&  \le & 1 + \frac{71}{252}(\nH + \mH - 2) \1 \\
& = & \frac{1}{252}(71\nH + 71\mH + 110). \hspace*{0.5cm} \Box
\end{array}
\]
\end{proof}

\section*{Acknowledgements}

 Research of the first and third author is supported by the
 Hungarian Scientific Research Fund NKFIH/OTKA
under the grant SNN 116095. Research of the second author is
supported in part by the South African National Research Foundation
and the University of Johannesburg.


\end{document}